\documentclass{amsart}
\usepackage{graphicx, amsmath, amsfonts, amssymb, hyperref, amsthm, comment, stmaryrd, cite}
\usepackage{algorithm2e}
\usepackage{listings}
\usepackage[shortlabels]{enumitem}% Required for inserting images
\theoremstyle{definition}
\newtheorem{thm}{Theorem}[section]
\newtheorem{prop}[thm]{Proposition}
\newtheorem{lem}[thm]{Lemma}
\newtheorem{dfn}[thm]{Definition}
\newtheorem{rmk}[thm]{Remark}
\newtheorem{exm}[thm]{Example}

\title{Ruelle's Zeta Function for non-Archimedean Rational Maps}

\author{Yunping Jiang and Chenxi Wu}

 

\begin{document}

\begin{abstract}
We studied the transfer operators defined over $\mathbb{C}_p$-valued analytic functions for subhyperbolic rational maps on $\mathbb{Q}_p$, and showed that the corresponding Ruelle's zeta functions are meromorphic on $\mathbb{C}_p$. We also used $\mathbb{R}$-valued transfer operators to study the shape of the corresponding Julia sets, and proved a Levin-Sodin-Yuditski type identity for general rational maps on $\mathbb{C}_p$. In all the results above, $\mathbb{Q}_p$ can be replaced with any non-Archimedean local field with characteristic $0$, and $\mathbb{C}_p$ the metric completion of its algebraic closure.
\end{abstract}

\maketitle

\section{Introduction}

The properties of Ruelle's zeta function of a complex quadratic polynomial map $f$ have been studied in~\cite{baladi2002dynamical}. In particular, under some mild conditions, one can express Ruelle's zeta function for weight $1/f'^2$ via the Levin-Sodin-Yuditski identity, which is related to the forward orbit of the critical points, and the zeta function is a ratio between entire functions when $f$ is quadratic like and subhyperbolic. 

To find non-Archimedean analogies of these results, we observe that the Levin-Sodin-Yuditski identity requires that the field involved be algebraically closed, so one should consider, for example, something like $\mathbb{C}_p$ instead of $\mathbb{Q}_p$; the study of subhyperbolic maps often need to assume the space involved to be locally compact, so one should consider, for example, something like $\mathbb{Q}_p$ and not $\mathbb{C}_p$. Analyticity of the weight function is an essential assumption for both, so one would have to look at functions that take values in the non-Archimedean field itself or its extensions, instead of $\mathbb{C}$-valued functions.

Following these observations, we have the following:

\begin{thm}[\protect{Analogy of \cite[Theorem A]{baladi2002dynamical}}]\label{lsy}
When $f$ is a polynomial map on $\mathbb{C}_p$ whose all critical points are non-degenerate and whose critical orbits on $\mathbb{P}^1(\mathbb{C}_p)$ are all disjoint. Then the following two formal power series are identical
%both the left hand side and the right hand side of the identity below are well defined, then
\[\exp\left(-\sum_n\frac{z^n}{n}\sum_{x\in \hbox{Fix}(f^n)}\frac{1}{(f^n(x))'((f^n(x))'-1)}\right)\]
\[=\det\left(1-\left[\sum_n\frac{z^n}{((f^n)''(c_i))(f^n(c_i)-c_j)}\right]_{i, j}\right),\]
here $c_i$ are the critical points of $f$, and $\hbox{Fix}(f^{n})$ is the set of fixed points of $f^{n}$. In particular, when both the left-hand side and the right-hand side are well-defined, the above equality holds.  
\end{thm}

\begin{thm}[\protect{Analogy of \cite[Theorem B]{baladi2002dynamical}}]\label{main2}
Let $f$ be a subhyperbolic rational map on $\mathbb{P}^1(\mathbb{Q}_p)$. The set of all repelling fixed points of $f^{n}$ is $\hbox{Fix}(f^n)\cap J(f)$.  Suppose $\beta>0$ is a real number and $\psi$ is a bounded and locally analytic function. Here, by ``locally analytic'' we mean that at any point $x\in J(f)$, there is a neighborhood on which $\psi$ can be written as a convergent power series with coefficients in $\mathbb{C}_p$). Then
\[\exp\left(\sum_{n=1}^\infty \left(\sum_{x\in \hbox{Fix}(f^n)\cap J(f)}\frac{\prod_{i=0}^{n-1}\psi(f^i(x))}{ ((f^n)'(x))^\beta(1-((f^n)'(x))^{-1})}\right)\frac{z^n}{n}\right)\]
is a quotient between two entire functions over $\mathbb{C}_p$.
\end{thm}

The expression in Theorem \ref{main2} is convergent only in a disk of radius $r\geq 0$ centered at $0$. As in the study of the Riemann zeta function, a fundamental problem in the study of dynamical zeta functions is whether this expression admits an analytic continuation to a transcendental meromorphic function on ${\mathbb C}$ or ${\mathbb C}_{p}$.  

For a Markov analytic dynamical system in the
$\mathbb{C}$ setting, Ruelle \cite{ruelle2002dynamical} proved, based on the nuclear operator theory developed by Grothendieck, that the expression admits an analytic continuation to a quotient of  two entire functions on ${\mathbb C}$. For a post-critically finite quadratic polynomial in the $\mathbb{C}$ setting, the authors of \cite{baladi2002dynamical} constructed a weak Markov partition based on the Yoccoz puzzle technique and, again using Grothendieck’s nuclear operator theory, showed that the expression extends analytically to a quotient of entire function on ${\mathbb C}$. It is natural to ask whether a similar result holds in the non-Archimedean setting. A key result proved in \cite{hsia2025zeta} (see Proposition \ref{markov}) provides an opportunity to address this question. Proposition \ref{markov} states that for a hyperbolic rational map on $\mathbb{Q}_p$, there exists a Markov partition, and furthermore, for a subhyperbolic rational map on $\mathbb{Q}_p$, there exists a weak Markov partition. 

The main point of Theorem \ref{main2} is that Grothendieck’s nuclear operator theory also works in the non-Archimedean setting and which combines the ideas of \cite{baladi2002dynamical} and \cite{hsia2025zeta}, the expression admits an analytic continuation to an quotient of two entire functions on $\mathbb{C}_p$.

These results are stated for $\mathbb{Q}_p$ or $\mathbb{C}_p$ just for convenience. All results for $\mathbb{Q}_p$ can be generalized to any non-Archimedean locally compact field with characteristic $0$ (for example, finite extensions of $\mathbb{Q}_p$),  while all results for $\mathbb{C}_p$ can be generalized to any non-Archimedean field which is algebraically closed.

On the other hand, the classical Ruelle's zeta function for $\mathbb{C}$-valued functions would be needed for the calculation of the Hausdorff dimension of the Julia set, and due to the nature of non-Archimedean geometry such calculation would be much simpler than the Archimedean case, which we will outline towards the end of the paper. 

In Section 2, we will introduce notation and provide some background for the main results. In Section 3, we will prove Theorem~\ref{lsy}, which primarily uses the fact that $\mathbb{C}_p$ is algebraically closed. In Section 4, we will study Ruelle's zeta function for subhyperbolic rational maps on $\mathbb{Q}_p$ and prove Theorem~\ref{main2}, by utilizing the Markov coding in~\cite{hsia2025zeta} and results from non-Archimedean functional analysis. And lastly, in Section 5, we will calculate the Hausdorff dimension for the Julia sets of subhyperbolic rational maps on $\mathbb{Q}_p$ (Theorem \ref{dim}), via the Markov coding in~\cite{hsia2025zeta}.

\medskip
\medskip
\noindent {\bf Acknowledgments.} The work is partially supported by PSC-CUNY awards, a PSC-CUNY enhanced award (66680-54), and Simons Foundation awards (523341, 942077, 850685). 

\medskip
\medskip

\section{Notification and Background}

Let $p$ be a prime number, $f$ be a rational map on $\mathbb{Q}_p$ (or $\mathbb{C}_p$), then $f$ defines a holomorphic self map on $\mathbb{P}^1(\mathbb{Q}_p)$ (or $\mathbb{P}^1(\mathbb{C}_p)$). We can define the concept of {\em Fatou sets} $F(f)$ and {\em Julia sets} $J(f)$ similar to the Archimedean setting, as the domain of equicontinuity and its complement, see \cite{benedetto2019dynamics}.

Following \cite{hsia2025zeta}, we have the following definition:

\begin{dfn}
We say $f$ is {\em hyperbolic} if there are no critical points of $f$ in $J(f)$. We say $f$ is {\em subhyperbolic} if all critical points of $f$ which are in $J(f)$ has finite forward orbit.
\end{dfn}

\begin{rmk}
A rational map $f$ is said to be post-critically finite (PCF) if the union of the forward orbits of its critical points is finite. Unlike the Archimedean setting, in the non-Archimedean setting the Julia set of a PCF rational map is always trivial. However, the Julia set of a subhyperbolic rational map in the non-Archimedean setting can be much more complicated and is therefore worthy of further investigation.
\end{rmk}

\begin{exm}
    An example of a subhyperbolic rational map is $z\mapsto \frac{z(z-1)^2}{3}$ on $\mathbb{Q}_3$. More examples and the description of their dynamics can be found in \cite{hsia2025zeta}.
\end{exm}

The following proposition proved in \cite{hsia2025zeta} becomes a key result enabling us to do further investigation in this paper.  The reader who is interested in the proof of this proposition can refer to \cite[pages 9-23]{hsia2025zeta}.

\begin{prop}[{\cite[Proposition 3.20]{hsia2025zeta}}]\label{markov}
Let $f$ be a rational map on $\mathbb{Q}_p$.
\begin{enumerate}
\item When $f$ is hyperbolic, $J(f)$ admits a finite Markov decomposition, where each Markov block is the intersection of $J(f)$ with a disc.
\item When $f$ is subhyperbolic, $J(f)$ admits a finite Markov decomposition where each Markov block is either the intersection of $J(f)$ with a disc (called {\rm a finite block}), or the forward orbit of the intersection of $J(f)$ with a disc under a scaling map that fixes a point in the forward orbit of a critical point (called {\rm an infinite block}).
\end{enumerate}
\end{prop}\qed

\begin{exm}(\cite[5.2]{hsia2025zeta})
    As an example, let $p>2$ be a prime, consider the polynomial map $z\mapsto \frac{z(z-1)^2}{8p}$, the Julia set has a finite Markov decomposition with 6 blocks: $\{0\}$, $\{1\}$, and four blocks $\alpha^{\pm}$, $\beta^{\pm}$ all consisting of the union of countably infinitely many discs, and the transition graph is 
    \begin{center}
    \includegraphics[scale=1]{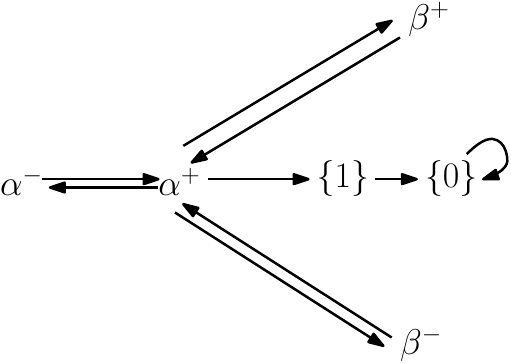}
    \end{center}
\end{exm}

\section{Levin-Sodin-Yuditski Type Identity}

\begin{proof}[Proof of Theorem \ref{lsy}]
To show this theorem, we expand both the left-hand side and the right-hand side as formal power series of $z$ and compare the coefficients. Each $\sum_{x\in \hbox{Fix}(f^n)}\frac{1}{(f^n(x))'((f^n(x))'-1)}$ on the left-hand side is a rational map of rational coefficients of all the roots of $f^n(x)=x$ as well as the coefficients of $f$ itself, and is also symmetric with respect to the roots of $f^n(x)=x$. Hence by algebra it is a rational map of rational coefficients of the coefficients of $f$ itself. Similarly, each $z^n$ coefficient on the right-hand side is a rational map of rational coefficients of the coefficients of $f$ as well. To show the left-hand side and the right-hand side are equal, one only need to show that the coefficients of $z^n$ on the left-hand side and the right-hand side are all identical as rational maps of the coefficients of $f$.

Now, replacing the field from $\mathbb{C}_p$ to $\mathbb{C}$, one sees that the $z^n$ coefficients on the left-hand side and right-hand side are exactly the same rational maps of the $N$ coefficients of $f$ as in the $\mathbb{C}_p$ case. Denote them as $C_n\in\mathbb{Q}(a_1, \dots, a_N)$ and $C'_n\in\mathbb{Q}(a_1, \dots, a_N)$, where $a_j$ are the coefficients of $f$. Since \cite{levin1991ruelle} showed that when $a_1, \dots, a_n$ are taking values in an open, hence Zariski dense subset of $\mathbb{C}^N$, we have $C_n(a_1, \dots, a_N)=C'_n(a_1, \dots, a_N)$, we have $C_n=C'_n$ for all $n$, which finished the proof.
\end{proof}

\section{\texorpdfstring{$\mathbb{C}_p$}{Cp}-valued zeta functions for rational maps on \texorpdfstring{$\mathbb{Q}_p$}{Qp}}

Now we focus on the dynamics of hyperbolic and subhyperbolic rational maps on $\mathbb{Q}_p$. The goal is to prove Theorem \ref{main2}.

\subsection{Nuclear Operators, Trace and Determinant}

Firstly, we recall some basic definitions and results in non-Archimedean functional analysis. These definitions and results can be found in \cite{schneider2013nonarchimedean}, and \cite{jiangnanjing} has their analogues in the Archimedean setting.

By a $\mathbb{C}_p$-Banach space we mean a $\mathbb{C}_p$ vector space with a complete norm $\|\cdot \|$ compatible with the norm on $\mathbb{C}_p$. For example, let $D\subseteq \mathbb{P}^1(\mathbb{C}_p)$ be an open disc, the space of $\mathbb{C}_p$ valued analytic functions on $D$, denoted as $\mathcal{A}(D)$, is a $\mathbb{C}_p$ Banach space via the $\sup$ norm $\|f\|=\sup_{z\in D}|f(z)|$.

\begin{dfn}\label{defnuclear}
A linear map $L$ between two $\mathbb{C}_p$ Banach spaces $A$ and $B$ is called {\em nuclear}, if $L=\sum_{i=1}^\infty \lambda_i y_i\otimes x_i$, where $x_i\in A^*$, $y_i\in B$, $\|x_i\|=\|y_i\|=1$, and $\lim_{i\rightarrow\infty} \lambda_i=0$. If, furthermore, $|\lambda_i|\leq O(\alpha^i)$ for some $0<\alpha<1$, we call $L$ $0$-nuclear.
\end{dfn}

\begin{rmk}\label{comp}
The compositions between the $0$-nuclear operator and the bounded operator are still $0$-nuclear, and the sums of the nuclear operators are also nuclear. Both follow immediately from the definition.
\end{rmk}

\begin{thm}[\protect{\cite[Proposition 21.1]{schneider2013nonarchimedean}}]
The {\em trace} of a $\mathbb{C}_p$ nuclear operator $L=\sum_i \lambda_i y_i \otimes x_i$ can be defined as $tr(L)=\sum_i \lambda_i y_i(x_i)$. The trace is independent of the choice of bases $\{x_i\}$. Due to non-archimedean property any nuclear operator is $0$-nuclear, hence the sum always exists.
\end{thm}\qed

\begin{dfn}
Let $L$ be a $\mathbb{C}_p$ nuclear operator, define $\det(I-tL)$ as
\[\exp\left(\sum_{n=1}^\infty \frac{tr(L^n)z^n}{n}\right)=\frac{1}{\det(I-zL)}\]
\end{dfn}

\begin{thm}\label{entire}
If $L$ is $0$-nuclear, then $\det(I-zL)$ is an entire function on $\mathbb{C}_p$.
\end{thm}

\begin{proof}
    Let $L_m=\sum_{i=1}^m\lambda_iy_i\otimes x_i$. Then $\det(I-zL_m)$ is a degree $m$ polynomial, and by linear algebra we have
    \[\exp\left(\sum_{n=1}^\infty \frac{tr(L_m^n)z^n}{n}\right)=\frac{1}{\det(I-zL_m)}.
    \]
    Take the limit $m\rightarrow\infty$, we get that $\det(I-zL)=\lim_{m\rightarrow\infty}\det(I-zL_m)$ where the convergence is in the sense of coefficients.
    
    For any $z\in\mathbb{C}_p$, as $m\to\infty$,
    \[\det(I-zL_{m+1})-\det(I-zL_m)=O(\alpha^m)\]
    for some $0<\alpha<1$. Hence $\{\det(I-zL_m)\}$ converges to an entire function on $\mathbb{C}_p$.
\end{proof}

\subsection{Transfer Operator and Zeta Function, Hyperbolic Case}\label{hypzeta}

Let $\psi$ be a nowhere vanishing locally analytic function on $\mathbb{P}^1(\mathbb{Q}_p)$. Recall that for any disc $B$ in $\mathbb{P}^1(\mathbb{C}_p)$, we let $\mathcal{A}(B)$ be the space of $\mathbb{C}_p$-valued analytic functions on $B$.

When $f$ is hyperbolic, by the proof of Proposition \ref{markov}, in particular the splitting lemma \cite[Proposition 3.20]{hsia2025zeta}, the Julia set $J(f)$ has a finite Markov decomposition by its intersections with finitely many disjoint discs $D_1, \dots, D_n$ in $\mathbb{P}^1(\mathbb{Q}_p)$, such that $f$ when restricted to each $D_i$ is a uniform scaling and $\psi$ is analytic on each $D_i$. Each $D_i$ is the intersection of some open disc $B_i\in\mathbb{P}^1(\mathbb{C}_p)$ with $\mathbb{P}^1(\mathbb{Q}_p)$. By adjusting the radius, we can ensure that:
\begin{equation}\label{adj1}
    \text{If }B_j\subseteq f(B_i)\text{ then }\overline{B_j}\subseteq f(B_i)
\end{equation}

\begin{dfn}\cite{ruelle2002dynamical, jiangnanjing}
We define the {\em transfer operator} of weight $\psi$, on $\prod_i \mathcal{A}(B_i)$, as
\[\mathcal{L}\phi(z)=\sum_{f(y)=z}\psi(y)\phi(y)\]
\end{dfn}

\begin{lem}\label{restnucl}
Let $B$, $B'$ be two open discs centered at $0$, the radius of $B'$ is smaller than the radius of $B$, then the restriction map $i: \mathcal{A}(B)\rightarrow\mathcal{A}(B')$ is $0$-nuclear.
\end{lem}

\begin{proof}
    Without loss of generality, assume that $B$ is the unit disc, let $r\in\mathbb{C}_p$ such that $|r|$ equals the radius of $B'$. Then we can let $y_i=(x/r)^i$, $x_i=(x^i)^*$ (here $\{(x^i)^*\}$ is the dual basis of $\{x^i\}$), $\lambda_i=r^i$, $\alpha=|r|<1$ in Definition \ref{defnuclear}.
\end{proof}

\begin{thm}\label{zetadet}
The transfer operator for a hyperbolic map $f$ on $\mathbb{P}^1(\mathbb{Q}_p)$ is 0-nuclear, and
\begin{equation}\label{releq}
    \frac{1}{\det(I-z\mathcal{L})}=\exp\left(\sum_{n=1}^\infty \left(\sum_{x\in \hbox{Fix}(f^n)\cap J(f)}\frac{\prod_{i=0}^{n-1}\psi(f^i(x))}{1-((f^n)'(x))^{-1}}\right)\frac{z^n}{n}\right)
\end{equation}
\end{thm}

\begin{proof}(Analogous to \cite[Chapter 2]{jiangnanjing})
To show that $\mathcal{L}$ is nuclear, note that it is a finite sum of operators of the form 
$
r_{ij}: \mathcal{A}(B_j)\rightarrow \mathcal{A}(B_i),
$
$$
r_{ij}\phi(z)=\psi((f|_{B_i})^{-1}(z))\phi((f|_{B_i})^{-1}(z)).
$$
Hence by Statement \eqref{adj1}, Lemma \ref{restnucl} and Remark \ref{comp} it is $0$-nuclear.

To show Equation~\eqref{releq}, let $\mathcal{S}_{n, i}$ be $(n+1)$-sequences $(a_0,\dotsm a_n)$ starting and ending at $i$ ($a_0=a_n=i$), such that $B_{a_j}\subseteq f(B_{a_{j-1}})$. Then 
\[\mathcal{L}^n=\sum_i\sum_{(a_j)\in \mathcal{S}_{n, i}}r_{a_0a_1}\circ\dots\circ r_{a_{n-1}a_n}\]
For each $(a_j)\in\mathcal{S}_{n, i}$, let $x_{i, a}$ be the unique fixed point of $f|_{B_{a_{n-1}}}\circ \dots \circ f|_{B_{a_0}}$, and pick some $x_{j, a}\in B_j$ for each $j\not=i$. Then, for each $B_i$ we pick a basis $\{((z-x_i)/r_i)^m\}$ consisting of elements of unit length, take a union we get a basis of $\prod_i \mathcal{A}(B_i)$. Now the trace of $r_{a_1a_2}\circ\dots\circ r_{a_{n-1}a_n}$ under this basis equals
\[\sum_{m=0}^\infty \left(((z-x_{i, a})/r_i)^m\right)^*\left(r_{a_1a_2}\circ\dots\circ r_{a_{n-1}a_n}\right)\left(((z-x_{i, a})/r_i)^m\right)\]
\[=\sum_{m=0}^\infty \left(((z-x_{i, a})/r_i)^m\right)^*\left(\prod_{j=0}^{n-1}\psi\left(\left(f|_{B_{a_{n-1}}}\circ \dots \circ f|_{B_{a_j}}\right)^{-1}(z)\right)\right)\]
\[\left(\left(f|_{B_{a_{n-1}}}\circ \dots \circ f|_{B_{a_0}}\right)^{-1}(z)-x_i\right)^m/(r_i^m)\]
\[=\left(\prod_{j=0}^{n-1}\psi\left(\left(f|_{B_{a_{n-1}}}\circ \dots \circ f|_{B_{a_j}}\right)^{-1}(x_{i, a})\right)\right)\sum_{m=0}^\infty \left(\left(f|_{B_{a_{n-1}}}\circ \dots \circ f|_{B_{a_j}}\right)'(x_{i, a})\right)^{-m}\]
\[=\frac{\prod_{i=0}^{n-1}\psi(f^i(x_{i, a}))}{1-((f^n)'(x_{i, a}))^{-1}}\]

Now Equation \eqref{releq} follows from the fact that there is a one-to-one correspondence between $\hbox{Fix}(f^{n})\cap J(f)$ and the set of $(n+1)$-sequences in $\bigcup_i \mathcal{S}_{n, i}$.
\end{proof}

\begin{proof}[Proof of Theorem \ref{main2} for hyperbolic maps]
Now Theorem \ref{main2} for hyperbolic $f$ follows from the theorem above and Theorem \ref{entire}.
\end{proof}

\begin{rmk}
Because $\mathbb{C}_p$ is not an ordered field, there is no version of the Ruelle-Perron-Frobenuous theorem (refer to~\cite{jiangmax,jiangnanjing}) for our setting. In particular, there may be more than one leading eigenvalues or eigenvectors for $\mathcal{L}$. To see this, let $\psi\equiv 1$, make four discs in $\mathbb{Q}_p$ of identical size, define linear functions on them sending the first two to their union, the other two to their union as well, such that the derivative on the first and third disc and the second and fourth discs are the same. Now apply the gluing procedure as in \cite[Section 4.2.4]{hsia2025zeta} (which is based on \cite{nopal2025gluing}), we get that $\det(I-z\mathcal{L})$ can be decomposed to two factors corresponding to the first two, and the third and fourth disc, whose coefficients are very close to one another, which result in at least two leading eigenvectors. 
\end{rmk}

\subsection{Subhyperbolic maps}\label{secsubhyp}

\begin{proof}[Proof of Theorem \ref{main2} for subhyperbolic maps]

Now we deal with the case when $f$ is subhyperbolic. By the description of the Markov decomposition in \cite{hsia2025zeta}, for example \cite[Definition 3.15, Proposition 3.20]{hsia2025zeta}, where each Markov block is called ``the closure of the union of a $T$-class'' in the language of \cite{hsia2025zeta}. Each $T$-class consists of either a single disc on $J(f)$, or a set of infinitely many disjoint discs in $J(f)$, which are sent to one another by a scaling map $T$ fixing some point $x$. In the latter case we call it an infinite block, and $x$ is called the center.

For each infinite block $C_j$, if the center $a_j$ of the corresponding scaling map is a critical point which is not in the forward image of some other critical point, then keep it, and pick a disc $B_j$ centered at $a$ whose radius is slightly larger than the radius of this block. If not, let $C'_j$ be one of its preimage in the earliest critical point (the critical point not in the forward orbit of any other critical point) in the backward path of $a_j$, and define $B_j$ similarly with $C_j$ replaced by $C'_j$. For finite blocks we define $B_j$ as in Section \ref{hypzeta}. Following \cite{baladi2002dynamical} we call $\{B_j\}$ the {\rm dynamical charts}. For each $B_j$, let $\mathcal{A}(B_j)$ be the space of $\mathbb{C}_p$-valued analytic functions on $B_j$, and define the transfer operator on $\prod_i \mathcal{A}(B_i)$ defined as
\[\mathcal{L}_\beta\phi(z)=\sum_{g(y)=z}\psi(y)(g'(y))^{-\beta}\phi(y)\]
Here $\beta>0$ is a real number.
Now by the definition of the dynamical charts, the weight function $\psi(y)(g'(y))^{-\beta}$ is always regular, hence by adjusting the radius appropriately we can write it as a linear combination of $0$-nuclear operators, which makes $\mathcal{L}_\beta$ $0$-nuclear. Hence, by Theorem \ref{entire} when $\psi$ is bounded and locally analytic on the disjoint union of $B_j$,
\[\frac{1}{\exp\left(\sum_{n=1}^\infty \frac{tr(\mathcal{L}_\beta^n)z^n}{n}\right)}=\det(I-z\mathcal{L}_\beta)\]
is an entire function.

To relate $\exp\left(\sum_{n=1}^\infty \frac{tr(\mathcal{L}_\beta^n)z^n}{n}\right)$ with the dynamical $\zeta$ function as in the right hand side of Equation \eqref{releq} in the statement of Theorem \ref{zetadet}, one observes that by \cite[Definition 3.15]{hsia2025zeta}, any $(n+1)$-sequence $(i_0, \dots, i_n)$ such that  $C_{i_{j+1}}\subseteq f(C_{i_j})$ must be in one of the following three cases:
\begin{enumerate}
\item All blocks $C_{i_j}$ are finite. Then such a sequence corresponds to a unique periodic point of period $n$.
\item All blocks $C_{i_j}$ are infinite and the transition maps are all of degree $1$. Such a sequence corresponds to a repelling periodic point of period $n$ which lies in the forward orbit of some critical point. There are only finitely many such periodic point under the subhyperbolic assumption. 
\item Some blocks $C_{i_j}$ are infinite and some are finite. Such a sequence corresponds to a unique periodic point of period $n$ as well.
\end{enumerate}

It is easy to see that all periodic points of period $n$ are accounted for by some admissible sequence, and except for those in the second case, they all correspond to a unique sequence.

Let $\mathcal{P}$ be the indices of the finitely many infinite pieces centered at the finitely many periodic points in the forward orbit of the critical points, and let $\mathcal{P}'$ be the set of these periodic points. For every $l\in \mathcal{P}$ or $a\in\mathcal{P}'$, let $d_l$ be the degree of the dynamical coordinate of $B_l$, let $n_l$ (or $n_a$) be the shortest period. For each period in $\mathcal{P}$ or $\mathcal{P'}$, pick an unique representative, put them together we get $\mathcal{Q}$ and $\mathcal{Q}'$ respectively. By the same argument as the proof of Equation \ref{releq} above, while accounting for the multiple counting due to the second case, we see that
\begin{equation}\label{eqsubhyp}
\begin{aligned}
&\exp\left(\sum_{n=1}^\infty \left(\sum_{x\in \hbox{Fix}(f^n)\cap J(f)}\frac{\prod_{i=0}^{n-1}\psi(f^i(x))}{1-((f^n)'(x))^{-1}}\right)\frac{z^n}{ n}\right)\\
=&\frac{1}{\det(I-z\mathcal{L}_\beta)}\cdot\frac{\exp\left(\sum_{n=1}^\infty \left(\sum_{a\in\mathcal{P}', n_a|n}\frac{\prod_{i=0}^{n-1}\psi(f^i(a))}{((f^n)'(a))^{\beta}(1-((f^n)'(a))^{-1})}\right)\frac{z^n}{ n}\right)}{\exp\left(\sum_{n=1}^\infty \left(\sum_{l\in\mathcal{P}, n_l|n}\frac{\prod_{i=0}^{n-1}\psi(f^i(a_l))}{((f^n)'(a_l))^{\beta/d_l}(1-((f^n)'(a_l))^{-1/d_l})}\right)\frac{z^n}{n}\right)}
\end{aligned}
\end{equation}
and the extra factor
\[\frac{\exp\left(\sum_{n=1}^\infty \left(\sum_{a\in\mathcal{P}', n_a|n}\frac{\prod_{i=0}^{n-1}\psi(f^i(a))}{((f^n)'(a))^{\beta}(1-((f^n)'(a))^{-1})}\right)\frac{z^n}{ n}\right)}{\exp\left(\sum_{n=1}^\infty \left(\sum_{l\in\mathcal{P}, n_l|n}\frac{\prod_{i=0}^{n-1}\psi(f^i(a_l))}{((f^n)'(a_l))^{\beta/d_l}(1-((f^n)'(a_l))^{-1/d_l})}\right)\frac{z^n}{n}\right)}\]
\[=\frac{\prod_{a\in\mathcal{P'}}\exp\left(\sum_{n=1}^\infty \frac{\left(\prod_{i=0}^{n_a-1}\psi(f^i(a))\right)^n}{((f^{n_a})'(a))^{n\beta}(1-((f^{n_a})'(a))^{-n})}\frac{z^{nn_a}}{nn_a}\right)}{\prod_{l\in\mathcal{P}}\exp\left(\sum_{n=1}^\infty \frac{\left(\prod_{i=0}^{n_l}\psi(f^i(a_l))\right)^n}
{((f^{n_l})'(a_l))^{n\beta/d_l}(1-((f^{n_l})'(a))^{-n/d_l})}\frac{z^{nn_l}}{nn_l}\right)}\]
\[=\prod_{j=0}^\infty\frac{\prod_{a\in\mathcal{Q'}}\exp\left(\sum_{n=1}^\infty \left(\prod_{i=0}^{n_a-1}\psi(f^i(a))\right)^n((f^{n_a})'(a))^{-n\beta-nj}\frac{z^{nn_a}}{n}\right)}{\prod_{l\in\mathcal{Q}}\exp\left(\sum_{n=1}^\infty \left(\prod_{i=0}^{n_l}\psi(f^i(a_l))\right)^n((f^{n_l})'(a_l))^{-n(\beta+j)/d_l}\frac{z^{nn_l}}{n}\right)}\]
\[=\prod_{j=0}^\infty\frac{\prod_{l\in\mathcal{Q}}\left(1-\left(\prod_{i=0}^{n_l}\psi(f^i(a_l))\right)((f^{n_l})'(a_l))^{-(\beta+j)/d_l}z^{n_l}\right)}
{\prod_{a\in\mathcal{Q'}}\left(1-\left(\prod_{i=0}^{n_a-1}\psi(f^i(a))\right)((f^{n_a})'(a))^{-\beta-j}z^{n_a}\right)}\]
is a ratio between entire functions (c.f. \cite[Equation (1.5)]{baladi2002dynamical}). Here 
$$
((f^{n_l})'(a_l))^{-\frac{\beta+j}{d_l}}
$$ 
is the derivative of the corresponding $g^{n_l}$ at $a_l$, where $g$ is defined by patching the $g_{ij}$ defined using the dynamical charts.

Now the conclusion of Theorem \ref{main2} follows from Equation \eqref{eqsubhyp} and the fact that $1/\det(I-z\mathcal{L}_\beta)$ is an entire function.
\end{proof}

\section{Transfer operator for real valued weights, Hausdorff dimension}

We can also do the classical Ruelle-Perron-Frobenuous transfer operator, with weight $|f'|^{-\beta}$, then this weight function is locally constant everywhere except for the singular points. In addition, if one considers the dynamical chart in Section \ref{secsubhyp}, then the weight function is constant on each Markov block. Hence, under the dynamical charts, $\det(I-t\mathcal{L}_\beta)$ is the characteristic polynomial of a finite matrix whose entries are either $0$ or powers of $\alpha^\beta$ where $\alpha=\|p\|$. Hence, we have:

\begin{thm}\label{dim}
When $f$ is a subhyperbolic rational map on $\mathcal{P}^1(\mathbb{Q}_p)$, and $0<\alpha<1$ such that $\|p\|=\alpha$, then the Hausdorff dimension of $J(f)$ (under the projective metric of $\mathcal{P}^1(\mathbb{Q}_p)$) is of the form $\log(\lambda)/\log(\alpha)$, where $\lambda$ is an algebraic number.
\end{thm}\qed

\bibliographystyle{plain}
\bibliography{refs}

\begin{thebibliography}{1}

\bibitem{baladi2002dynamical}
Viviane Baladi, Yunping Jiang, and Hans~Henrik Rugh.
\newblock Dynamical determinants via dynamical conjugacies for postcritically
  finite polynomials.
\newblock {\em Journal of Statistical Physics}, 108:973--993, 2002.

\bibitem{benedetto2019dynamics}
Robert~L Benedetto.
\newblock {\em Dynamics in one non-archimedean variable}, volume 198.
\newblock American Mathematical Soc., 2019.

\bibitem{hsia2025zeta}
Liang-Chung Hsia, Hongming Nie, and Chenxi Wu.
\newblock Zeta function and entropy for \hbox{non-Archimedean} subhyperbolic
  dynamics.
\newblock {\em arXiv preprint arXiv:2503.10018}, 2025.

\bibitem{jiangnanjing}
Yunping Jiang.
\newblock Nanjing lecture notes in dynamical systems part one: Transfer
  operators in thermodynamical formalism.
\newblock {\em FIM-Zurich Preprint}, 1998.

\bibitem{jiangmax}
Yunping Jiang.
\newblock A proof of the existence and simplicity of a maximal eigenvalue for
  \hbox{Ruelle-Perron-Frobenius} operators.
\newblock {\em Lett. Math. Phys.}, 48:211--219, 1999.

\bibitem{levin1991ruelle}
GM~Levin, ML~Sodin, and PM~Yuditski.
\newblock A \hbox{Ruelle} operator for a real julia set.
\newblock {\em Communications in Mathematical Physics}, 141(1):119--132, 1991.

\bibitem{nopal2025gluing}
V{\'\i}ctor Nopal~Coello and J~Rogelio P{\'e}rez-Buend{\'\i}a.
\newblock Gluing dynamics: $\varepsilon$-precision in solving a
  \hbox{non-Archimedean} inverse problem.
\newblock {\em Bolet{\'\i}n de la Sociedad Matem{\'a}tica Mexicana}, 31(2):52,
  2025.

\bibitem{ruelle2002dynamical}
David Ruelle.
\newblock Dynamical zeta functions and transfer operators.
\newblock {\em Notices Amer. Math. Soc.}, 49(8):887, 2002.

\bibitem{schneider2013nonarchimedean}
Peter Schneider.
\newblock {\em Nonarchimedean functional analysis}.
\newblock Springer Science \& Business Media, 2013.

\end{thebibliography}

\medskip
\medskip

\noindent Yunping Jiang: Department of Mathematics, Queens College of the City University of New York, Flushing, NY 11367-1597 and The Ph.D. Program in Mathematics, Graduate Center of the City University of New York; New York, NY 10016\\
Email: yunping.jiang@qc.cuny.edu

\medskip
\noindent Chenxi Wu: Department of Mathematics, University of Wisconsin-Madison, 480 Lincoln Drive, Madison, WI 53706, USA\\
Email: cwu367@math.wisc.edu

\end{document}